\documentclass[12pt]{amsart}
\usepackage{amsmath}
\usepackage{amssymb}
\newtheorem{theorem}{Theorem}
\newtheorem{proposition}[theorem]{Proposition}
\newtheorem{corollary}[theorem]{Corollary}
\newtheorem{lemma}[theorem]{Lemma}
\newenvironment{proof*}{\vskip 2mm\noindent {}}{\hfill $\Box$ \vskip 2mm}

\newcommand{\C}{{\mathbb C}}
\newcommand{\R}{{\mathbb R}}
\newcommand{\D}{{\mathbb{D}}}
\newcommand{\E}{{\mathcal E}}
\newcommand{\OO}{{\mathcal O}}
\def\Re{\operatorname{Re}}
\def\Im{\operatorname{Im}}
\def\eps{\varepsilon}
\def\wdt{\widetilde}
\def\phi{\varphi}

\def\Om{\Omega}

\def\bt{\beta}

 \def\la{\lambda}

\def\bc{\bold C}
\def\C{\Bbb C}
\def\D{\Bbb D}
\def\vt{\vert}

\def\al{\alpha}
\def\de{\delta}

\def\vt{\vert}
\def\va{\varphi}
\def\va{\varphi}

\title[Invariant metrics near non-semipositive points]
{Estimates for invariant metrics near non-semipositive boundary
points}

\author{Nguyen Quang Dieu, Nikolai Nikolov, Pascal J. Thomas}

\address{Department of Mathematics\\
Hanoi University of Education (Dai Hoc Su Pham Ha Noi)\\ Cau Giay, Tu Liem \\Hanoi, Viet Nam}
\email{dieu\_vn@yahoo.com}

\address{Institute of Mathematics and Informatics\\ Bulgarian Academy
of Sciences\\ Acad. G. Bonchev 8, 1113 Sofia,
Bulgaria}\email{nik@math.bas.bg}

\address{Universit\'e de Toulouse\\ UPS, INSA, UT1, UTM \\
Institut de Math\'e\-matiques de Toulouse\\
F-31062 Toulouse, France} \email{pthomas@math.univ-toulouse.fr}



\begin{document}
\begin{thanks} {This note was written during the stay as guest professors
of the first and second named authors
at the Paul Sabatier University, Toulouse in June and July, 2010.  The first named author is partially supported by the NAFOSTED program. The collaboration between the second and third named authors is
supported by the bilateral cooperation between CNRS and the Bulgarian Academy of Sciences. }
\end{thanks}

\begin{abstract} We find the precise growth of some invariant metrics near
a point on the boundary of a domain where the Levi form has at least one negative
eigenvalue.
\end{abstract}

\maketitle

\section{Behavior of the Azukawa and Kobayashi-Royden pseudometrics}
\label{behav}

Let $D\subset\C^n$ be a domain. Denote by  $C_D,$ $S_D,$ $A_D$ and
$K_D$ the Carath\'eo\-dory, Sibony, Azukawa and
Kobayashi(--Royden)
 metrics of $D,$ respectively (cf.~\cite{Jar-Pfl1}).
$K_D$ is known to be the largest holomorphically invariant metric.  Recall that the
{\it indicatrix} of a metric $M_D$ at a base point $z$ is
$$
I_z M_D:= \left\lbrace v \in T^{\C}_z D: M_D(z,v)<1 \right\rbrace .
$$
The indicatrices of $C_D$ and $S_D$ are convex domains, and the indicatrices of $A_D$ are pseudoconvex
domains. The larger the indicatrices, the smaller the metric. The Kobayashi--Buseman
metric $\widehat K_D$ is the largest invariant metric with convex indicatrices (they are the convex hulls of
the indicatrices of $K_D$). Since the indicatrices of $K_D$ are balanced domains and the envelope of holomorphy
of a balanced domain in $\C^n$ is a balanced domain in $\C^n$, we may define $\wdt K_D$ to be {\it the largest
invariant metric with pseudoconvex indicatrices,} i.e. $I_z\wdt K_D$ to be the envelope of holomorphy of
$I_z K_D$ for any $z\in D.$ Then
\begin{equation}
\label{order}
C_D\le S_D\le\min\{A_D,\widehat K_D\}\le\max\{A_D,\widehat K_D\}\le\widetilde K_D\le K_D.
\end{equation}

We list some properties of $\wdt K_D$ in Section \ref{propnew}, Propositions
\ref{list} and \ref{cont}.

Let $D\Subset\C^n$, and suppose that $a \in \partial D$ and that the boundary $ \partial D$
is $\mathcal C^2$-smooth in a neighborhood of $a$. We say that $a$ is {\it semipositive} if the restriction of the
Levi form  on the complex tangent hyperplane to $\partial D$ at $a$ has only non-negative eigenvalues.
A {\it non-semipositive point} $a$ is such that the above restriction has a negative
eigenvalue. This is termed a "non-pseudoconvex point" in \cite{For-Lee}.

Denote by $n_a$ and $\nu_a$ the inward normal and a unit complex normal vector to $\partial D$
at $a,$ respectively. Let $z\in n_a$ near $a$ and $d(z)=\mbox{dist}(z,\partial D)\ (=|z-a|).$ Note that for $\mathcal C^2$-smooth boundaries,
$d^2$ is also  $\mathcal C^2$-smooth is a neighborhood of $\partial D$
\cite{Kra-Park}.
Due to
Krantz \cite{Kra} and Fornaess--Lee \cite{For-Lee}, the following estimates hold:
$$K_D(z;\nu_a)\asymp(d(z))^{-3/4},\quad S_D(z;\nu_a)\asymp(d(z))^{-1/2},\quad  C_D(z;\nu_a)\asymp 1.$$
In fact, one may easily see that $C_D(z;X)\asymp |X|$ for any $z$ near $a.$
Denote by $ \langle X,Y \rangle$ the standard hermitian product of vectors in $\C^n$.

Our purpose is to show the following extension of \cite[Theorem 1]{For-Lee}.

\begin{proposition}\label{conc} If $a$ is a non-semipositive boundary point of
a domain $D\Subset\C^n,$
then
$$S_D(z;X)\asymp\widetilde K_D(z;X)\asymp\frac{|\langle \nabla d(z),X\rangle|}{{d(z)}^{1/2}}+|X|\quad\mbox{near }a,$$
and by \eqref{order} that estimate holds for $A_D$ and $\widehat K_D$ as well.
\end{proposition}
Note that it does not matter whether the Levi form at $a$ has one or more
negative eigenvalues.

Using the arguments in \cite{For-Lee},
and for the case (i) a reduction to the model case along the lines
of the argument given in the proof of Proposition \ref{kobeps} in
section \ref{proof2}, one may show that

\begin{proposition}
\label{lower}
(i) If $0\le\eps\le 1$ and $a$ is a $\mathcal C^{1,\eps}$-smooth boundary point of
a domain $D\Subset\C^n,$ then
$$S_D(z;X)\gtrsim\frac{|\langle\nu_{a'},X\rangle|}{(d(z))^{1-\frac{1}{1+\eps}}}+|X|\quad\mbox{ near }a,$$
where $a'$ is a point near $a$ such that $z \in n_{a'}$.

(ii) If $0\le\eps\le 1$ and $a$ is a  semipositive $\mathcal C^{2,\eps}$-smooth boundary point of
a domain $D\Subset\C^n,$ then
$$S_D(z;X)\gtrsim\frac{|\langle\nabla d(z),X\rangle|}{(d(z))^{1-\frac{1}{2+\eps}}}+|X|,\quad z\in n_a\mbox{ near }a.$$
\end{proposition}

Thus for $\mathcal C^{2,\eps}$-smooth boundaries,
Propositions \ref{conc} and \ref{lower} (ii) characterize the semipositive points in terms of
the (non-tangential) boundary behavior of any metric between $S_D$ and $\wdt K_D.$
In particular, if $D$ is pseudoconvex and $\mathcal C^{2,\eps}$-smooth, then there can be no
$\alpha < 1-\frac{1}{2+\eps}$ and $a \in \partial D$ such that $S_D(z;X) \lesssim d(z)^{-\alpha}|X|$
for $z\in n_a$ near $a$.
A similar characterization
in terms of $K_D$ can be found in \cite{Fu}.
\smallskip

\noindent{\bf Remark.} For the Kobayashi metric
$K_D$ itself, one cannot expect simple estimates similar to that in Proposition \ref{conc}. In \cite[Propositions 2.3, 2.4]{Fu}, estimates are given for $X$ lying in a cone around the normal
direction, i.e. $|\langle\nabla d(z),X\rangle| \gtrsim |X|$.
One may modify the proofs of those propositions to obtain
that for a non-semipositive boundary point $a$ of
a domain $D\Subset\C^2$ there exists $c_1>0$ such that if
$$
|\langle\nabla d(z),X\rangle| > c_1 d(z)^{3/8}|X| ,
$$
then
$$
K_D(z;X)\asymp\frac{|\langle\nabla d(z),X\rangle|}{(d(z))^{3/4}}\quad\mbox{near }a.
$$

At least when $n=2$,
the range of those
estimates can be expanded. Part (3) should hold for any $n \ge 2$, with a similar proof.

\begin{proposition}
\label{dieu}
\ {}
Let $D\Subset\C^2$ be a domain with $\mathcal C^2$-smooth boundary.
\begin{enumerate}
\item
If $a$ is a non-semipositive boundary point of
a domain $D\Subset\C^2,$ then
$$
K_D(z;X)\lesssim\frac{|\langle\nabla d(z),X\rangle|}{(d(z))^{3/4}}+|X|\quad\mbox{near }a.
$$
\item
There exists $c_0>0$ such that if
$|\langle\nabla d(z),X\rangle|<c_0 d(z)^{1/2}|X|,$
then
$$
K_D(z;X)\asymp |X|,
$$
while if $|\langle\nabla d(z),X\rangle|>c_0 d(z)^{1/2}|X|,$
then
$$
\liminf_{d(z)\to0} d(z)^{1/6} \frac{K_D(z;X)}{|X|} >0.
$$
\item
There exists $c_1>c_0$ such that if
$|\langle\nabla d(z),X\rangle|>c_1 d(z)^{1/2}|X|,$
then
$$
K_D(z;X) \asymp \frac{|\langle\nabla d(z),X\rangle|}{(d(z))^{3/4}}.
$$
\end{enumerate}
\end{proposition}

The fact that $c_1$ cannot be made arbitrarily
small already follows from \cite[p.~6, Remark]{Fu}.
Notice that this is one more (unsurprising) instance of discontinuity of the
Kobayashi pseudometric: when $z_\delta =a + \delta \nu_a$,
$X_\delta =c \delta^{1/2} \nu_a + u_a$, where $|u_a|=1,$ $\langle\nu_a,u_a\rangle=0$, then there is a critical value of $c$
below which $K_D(z_\delta;X_\delta)$ remains bounded and above
which it blows up ; and if $c$ is large enough,
$K_D(z_\delta;X_\delta)$  behaves as $\delta^{-1/4}$.

When $\partial D$ is not $\mathcal C^2$-smooth,
we can also give estimates on the growth of the Kobayashi
pseudometric for vectors relatively close to the complex tangent direction
to the boundary of the domain, in the spirit of Proposition \ref{lower} (i), with strictly stronger exponents. Those are the same exponents found by
Krantz \cite{Kra} for the Kobayashi pseudometric applied to the normal
vector.
This result, however, is about vectors which have to make
some positive angle with the normal vector, but may not quite be
orthogonal to it,
 and applies
(for $\eps<1$) to domains which are slightly larger than
those considered by Krantz.

\begin{proposition}
\label{kobeps}
Let $0<\eps\le 1$, and a domain $D\Subset\C^2$ with
$\mathcal C^{1,\eps}$-smooth boundary.
Let $a \in \partial D$ and $z \in D$, close enough to $a$ such that
$a' \in \partial D$ is
a point near $a$ such that $z \in n_{a'}$ ($a'$ is not unique in general).
Then if
$|\langle\nu_{a'},X\rangle|>c_2 d(z)^{\eps/(1+\eps)} |X|$
and $|\langle\nu_{a'},X\rangle|< (1- c_3) |X|$
for some $c_2, c_3 >0$, then
$$
K_D(z;X)\gtrsim
\frac{|\langle\nu_{a'},X\rangle|}{(d(z))^{1-\frac{1}{2(1+\eps)}}}
\quad\mbox{ near }a.
$$

\end{proposition}

\section{Proof of Proposition \ref{conc}}
\label{proof1}

The main point in the proof of Proposition \ref{conc} is an upper estimate for
$\widetilde K_G$ on the model domain
$$G_\varepsilon =\mathbb B_n(0,\eps)\cap\{z=(z_1,z_2,z')\in\C^n:0>r(z)=\Re z_1-|z_2|^m+q(z')\},$$
where $\eps>0,$ $m\ge 1$ and $q(z')\lesssim|z'|^k,$ $0<k\le m.$

\begin{proposition}
\label{model}
If $\delta>0$ and $P_\delta=(-\delta,0,0'),$ then
$$
\wdt K_{G_\varepsilon}(P_\delta;X)\lesssim|X_1|\delta^{\frac{1}{m}-1}+|X_2|+|X'|\delta^{\frac{1}{m}-\frac{1}{k}}.
$$
\end{proposition}

Estimates for the Sibony and Kobayashi metrics on some model domains can be found in
\cite{For-Lee,Fu}.

\begin{corollary}
\label{cor} If $|q(z')|\lesssim|z'|^m,$ then
$$
S_{G_\varepsilon}(P_\delta;X)\asymp\wdt K_{G_\varepsilon}(P_\delta;X)\asymp|X_1|\delta^{\frac{1}{m}-1}+|X|.
$$
\end{corollary}

This corollary shows that the estimates in Proposition \ref{lower} are sharp.
\smallskip

\noindent{\it Proof of Corollary \ref{cor}.}
It follows by \cite[Remark 4,5]{For-Lee} that if $-q(z')\lesssim|z'|^m,$ then
\begin{equation}\label{sib}
S_{G_\varepsilon}(z;X)\gtrsim|X_1|\delta^{\frac{1}{m}-1}+|X|.
\end{equation}
Proposition \ref{model} implies the opposite inequality
$$S_{G_\varepsilon}(z;X)\le\wdt K_{G_\varepsilon}(z;X)\lesssim|X_1|\delta^{\frac{1}{m}-1}+|X|.\hfill\qed$$
\smallskip

\noindent{\it Proof of Proposition \ref{conc}.}
We may assume that $a=0$ and that the inward normal to $\partial D$ at $a$ is
$\left\lbrace \Re z_1<0, \Im z_1=0, z_2=0, z'=0 \right\rbrace $ and that $z_2$ is a
pseudoconcave direction. After dilatation
of coordinates and a change of the form $z\mapsto(z_1+cz_1^2,z_2,z'),$ we may get
$G_\eps\subset D$ for some $\eps>0,$ $m=2$ and $q(z')=|z'|^2.$ Then, by Proposition \ref{model},
$$
\wdt K_D(z;X)\le\wdt K_{G_\eps}(z;X)\lesssim\frac{|\langle\nabla d(z),X\rangle|}{{d(z)}^{1/2}}+|X|$$
if $z$ is small enough and lies on the inward normal at $a.$ Varying $a,$ we get the estimates for any
$z$ near $a.$  A similar argument together with (\ref{sib}) and a localization principle for the
Sibony metric (see \cite{For-Lee}) gives the opposite inequality
$$\widetilde K_D(z;X)\ge\widetilde S_D(z;X)\gtrsim\frac{|\langle\nabla d(z),X\rangle|}{{d(z)}^{1/2}}+|X|.\hfill\qed$$
\smallskip

\noindent{\it Proof of Proposition \ref{model}.} For simplicity, we assume that
$\eps=2$ and $q(z')\le|z'|^k,$ where $|\cdot|$ is the sup-norm
(the proof in the general case is similar).

It is enough to find constants $c,c_1>0$ such that for $0<\delta\ll 1,$
$$c_1\delta^{1-\frac{1}{m}}\D\times\D\times c\delta^{\frac{1}{k}-\frac{1}{m}}\D^{n-2}
\subset I_\delta:=I_{P_\delta}\wdt K_{G_\varepsilon},$$
where $\D$ denotes the unit disk in $\C$.

Take $X\in\C^n$ with $|X_2|=1,$ $|X_1|\le c_1\delta^{1-\frac{1}{m}},$
$|X'|\le c\delta^{\frac{1}{k}-\frac{1}{m}},$ and
set
$$\phi(\zeta)=P_\delta+\zeta X,\quad\zeta\in\D.$$
If $c<1$ and $0<\delta\ll 1,$ then $\phi(\D)\Subset\Bbb B_n(0,2).$
On the other hand,
$$
r(\phi(\zeta))<-\delta+|\zeta|.|X_1|-|\zeta|^m+|\zeta|^k|X'|^k.
$$
It follows that if $|\zeta|<\delta^\frac{1}{m},$ then
$r(\phi(\zeta))<(c_1+c^k-1)\delta,$ and if $|\zeta|\ge\delta^\frac{1}{m},$ then
$r(\phi(\zeta))<(c_1+c^k-1)|\zeta|^m.$ So, choosing $c_1=c^k<\frac{1}{2},$ we get
$\phi(\D)\Subset G$
and hence
$c_1\delta^{1-\frac{1}{m}}\overline{\D}\times\partial\D\times
c\delta^{\frac{1}{m}-\frac{1}{k}}\overline\D^{n-2}\subset I_\delta$.

Finally, using that $\{0\}\times\overline{\D}\times\{0'\}\subset I_\delta$ and that
$I_\delta$ is a pseudoconvex domain, we obtain the desired result by
Hartog's phenomenon.\qed
\smallskip

\section{Proof of Propositions \ref{dieu} and \ref{kobeps}}
\label{proof2}

\begin{proof*}{\it Proof of Proposition \ref{dieu}.}
As in the previous section, for $d(z)$ small enough, $z$ will
belong to the normal to $\partial D$ going through the point closest to $z$, which we take as the origin.
We make a unitary change of variables to have a new basis
$(\nu_a,u_a)$ of
vectors normal and parallel to $\partial D$, respectively.
Using different dilations
along the new coordinate axes
and the localization property of the
Kobayashi pseudometric, we can reduce  Proposition \ref{dieu}
to the following.
\end{proof*}

\begin{lemma}
\label{lemkob}
Let $G:=\{(z, w) \in \C^2: \Re z<\vert w\vert^2\} \cap \D^2$, where $\D$ is the unit disk in $\C$.
 Let
$P_\de:=
(-\de, 0) \in G, 0<\de<1$ and $\nu=(\al, \beta)$ be a vector in $\C^2.$
Then there exists   $\de_0=\de_0(\nu)>0$ such that for any $\de <\de_0$,
\begin{enumerate}
\item
 If $\vert \al \vert < 2 \sqrt 2 \de^{1/2} \vert \beta\vert$, then
$$K_{G} (p_\de, \nu) = \vert \beta\vert ;$$
while if $c_0:=\liminf_{\de\to0}|c_\delta|> 2 \sqrt 2$, there exists
$\gamma(c_0) >0$ such that $\liminf_{\de\to0} \delta^{1/6} K_G ((-\delta, 0) ; (c_\delta \delta^{1/2}, 1)) \ge \gamma(c_0)$.
\item
 If $\vert \al \vert \ge 2  \de^{1/2} \vert \beta\vert$ then
$$K_{G} (p_\de, \nu) \le \sqrt{2} \frac{\vert \al\vert}{\de^{3/4}} .$$
\item
 If $\vert \al\vert >7\de^{1/2} \vert \beta\vert$
then
 $$ K_G (p_\de, \nu) \ge \frac1{38} \frac{\vert \al\vert}{\de^{3/4}}.$$
 \end{enumerate}
\end{lemma}
\begin{proof}

(1).
By the Schwarz lemma we have $K_{G} (p_\de, \nu) \ge \vert \beta\vert$ for every $\de, \nu$.
Conversely, let $c:=\frac{|\al|}{\de^{1/2} \vert \beta\vert}<2\sqrt 2$.
Consider an analytic disk $\Phi :\C \to \C^2$, $ \Phi(t) =\left(f(t),g(t) \right)
= \left(-\de+ \al t - \frac{\al^2}{8\de} t^2, \beta t\right)$.

It will be enough to
show that $\Phi(t)\in G$ for $|t|< 1/|\beta|$.  Clearly $g(t)\in \D$.
Since $|\al t |< 2\sqrt 2 \de^{1/2}$
and $\left| \frac{\al^2}{8\de} t^2\right| < \frac{c^2}8 <1$, for $\de_0$ small enough
we have $f(t)\in \D$.

Now let $\al = |\al|e^{i\theta}$, and define $x,y\in \R$ by
$t= \de^{1/2} (x+iy) e^{-i\theta}/|\beta|$. Then
\begin{multline*}
\frac1\de \left(|g(t)|^2 -\Re f(t)\right)  =
\left( 1 +\frac{c^2}{8}\right) x^2 - c x + \left( 1 -\frac{c^2}{8}\right) y^2 + 1 \\
=
\left( 1 +\frac{c^2}{8}\right) \left( x - \frac{c}{2\left( 1 +\frac{c^2}{8}\right)}\right)^2
+ \left( 1 -\frac{c^2}{8}\right) y^2  + \frac{4 -\frac{c^2}{2}}{4 +\frac{c^2}{2}} > 0.
\end{multline*}

Observe that if $\Phi =(f,g)  \in \mathcal O( \D, G) $, then
$\Phi^\theta  \in \mathcal O( \D, G) $ with
$$
\Phi^\theta (\zeta):= \left( f(e^{i\theta}\zeta) , e^{-i\theta} g(e^{i\theta}\zeta) \right) ,
$$
and $(\Phi^\theta  )'(0) = (e^{i\theta}f'(0), g'(0))$.
So we may assume $c>0$.

If $c>2\sqrt2$, recall that
$$
K_G(p;X)^{-1} = \sup \left\{
r >0 : \exists \varphi \in \mathcal O(D(0,r), G): \varphi(0)=p, \varphi'(0)=X
\right\}.
$$
Suppose that there exists $\gamma \in (0,1)$ and a
sequence $(\delta_j)\to 0$, $c_j>0$ with $\liminf_j c_j=c_0$ such that
$$
k_j:=K_G ((-\delta_j, 0) ; (c_j\delta_j^{1/2}, 1)) \le \gamma \delta_j^{-1/6} .
$$
Choose $r_j$ such that $\gamma^{-1} \delta_j^{1/6} < r_j < 1/k_j$. Let
$\varphi_j (\zeta) = (f_j(\zeta),g_j(\zeta))$ be as in the definition.
>From now on we drop the indices $j$.

Write
$$
f(\zeta) = \sum_{k\ge 0} a_k \zeta^k, \quad
g(\zeta) = \sum_{k\ge 0} b_k \zeta^k.
$$
Since $G\subset \D^2$, the Cauchy estimates
imply $|a_k|, |b_k| \le r^{-k}$. Suppose henceforth that
$|\zeta| \le r/2$. Then
$$
f(\zeta) = -\delta + c\delta^{1/2} \zeta + a_2 \zeta^2 + \sum_{k\ge 3} a_k \zeta^k,
$$
and $\left| \sum_{k\ge 3} a_k \zeta^k \right| \le 2 r^{-3} |\zeta|^3$. Likewise,
$$
|g(\zeta)|^2 = |\zeta|^2 \left| 1 + \sum_{k\ge 2} b_k \zeta^{k-1}
\right|^2,
\quad \left| \sum_{k\ge 2} b_k \zeta^{k-1} \right| \le 2 r^{-2} |\zeta|,
$$
so, whenever $|\zeta|\le r^2$,
 $|g(\zeta)|^2 \le |\zeta|^2 + 8 r^{-2} |\zeta|^3$. All together,
using the definining function of $G$,
$$
-\delta + \Re \left( c\delta^{1/2} \zeta + a_2 \zeta^2 \right)
\le
|\zeta|^2 + 2 \gamma^{3} \delta^{-1/2} |\zeta|^3
+ 8 \gamma^{2} \delta^{-1/3} |\zeta|^3
\le
|\zeta|^2 + 10 \gamma^{2} \delta^{-1/2} |\zeta|^3.
$$
Now set $\zeta = \delta^{1/2} e^{i\theta}\in D(0,r^2)$ for $j$ large enough.
We can choose $\theta \in \left[ -\frac\pi4, \frac\pi4\right]$ so that
$\Re(a_2 e^{2i\theta}) \ge 0$. We have
$$
-\delta + \frac{c}{\sqrt2} \delta \le
-\delta + \Re \left( c\delta^{1/2} \zeta + a_2 \zeta^2 \right)
\le
\delta + 10 \gamma^{2} \delta ,
$$
which implies $\gamma \ge \left(\frac1{10} \left(  \frac{c_0}{\sqrt2} - 2 \right) \right)^{1/2}>0$.

(2).
We proceed as in the first case of (1) with
$ \Phi(t)
= \left(-\de+\la\al t,\la\beta t +\frac{t^2}2\right) \in \D^2$ for
$\de_0$ small enough and $|\la \al|, |\la \bt| <1/2$.
Then
 $ \Phi(t) \in G$ if and only if
 $$ -\de+  \left|\la\al t \right| <\Bigl|\la\beta t+\frac{t^2}2\Bigr|^2,
 \quad \forall t \in \D,
 $$
 which is true when
 $$\frac{\vert t\vert^4}{4} -\vert \la \beta\vert \vert t\vert^3 >
 -\de+\vert \la \al\vert \vert t\vert, \mbox{ i.e. }
 \frac{\vert t\vert^4}{4}+ \de > \vert \la \beta\vert \vert t\vert^3
 +\vert \la \al\vert \vert t\vert.
 $$
 If we now assume $|\la| < \frac1{\sqrt{2}} \frac{\de^{3/4}}{\vert \al\vert}$,
 using the fact that $a^4+b^4 \ge a^3b+ab^3$ for any $a,b \ge 0$,
 $$
 \frac{\vert t\vert^4}{4}+ \de >
 \frac{\vert t\vert^3}{2 \sqrt{2}} \de^{1/4} + \frac{\vert t\vert}{ \sqrt{2}} \de^{3/4}
\ge   \vert t\vert^3 \frac{|\la \al|}{2\de^{1/2}}+\left|\la\al t \right|,
 $$
 and the assumption on $|\al|$ gives the required inequality.

 (3).
 When $|\al| \ge C_0 |\beta|,$ this follows from the results of Fu, as
 explained in the Remark after Proposition \ref{lower}.
For $|\al| \le C_0 |\beta|,$ this is a special case of Lemma \ref{modeps} below.
\end{proof}

\begin{proof*}{\it Proof of Proposition \ref{kobeps}.}

For any $z\in D$, the function $f_z (y)=|z-y|$, $y\in \partial D$,
must attain its minimum.
Let $U_0$ be an open neighborhood of $a$. Since $\partial D \setminus U_0$
is closed,
if $z \in D \cap U_1$, where $U_1$ is a small enough neighborhood of $a$,
then $f_z$ will assume its minimum in $U_0\cap \partial D$. Let $a'$ be a
point where this minimum is attained. Since $f_z$ is
$\mathcal C^1$-smooth outside of $\partial D$ and $\nabla f_z(y)$
is parallel to $y-z$, by
Lagrange multipliers the outer normal vector $\nu_{a'}$
is parallel to $z-a'$.
Since the distance is minimal, the semi-open segment $[z,a')$ must lie inside $D$, therefore $z \in n_{a'}= a' + \R_-^* \nu_{a'}$.

By taking $a'$ as our new origin and
making a unitary change of variables, we may assume that locally
$D = \{ \zeta : \Re \zeta_1 < O(|\zeta_2|^{1+\eps} + |\Im \zeta_1|^{1+\eps}) \}$, so that after appropriate dilations we may
assume that $D\cap U_0 \subset \Om_\xi$, the model domain used in the following lemma, with $\xi=1+\eps$. We use the localization property
of the Kobayashi-Royden pseudometric.
The constants implied in the "$O$" above depend only on the neighborhood $U_0$ of $a$.
To get uniform constants, we cover $\partial D$ by a finite number of
neighborhoods of the type $U_1$.
\end{proof*}

\begin{lemma}
\label{modeps}
Let
$$\Om_\xi:=\{(z, w) \in \bc^2: \Re z<\vert w\vert^\xi + |\Im z|^\xi
\} \cap \D^2,$$ where  $\xi >1$.
 Let
$p_\de:=
(-\de, 0) \in \Om_\xi, \de>0$ and $\nu=(\al, \beta)$ be a vector in $\bc^2.$
Let $C_0>0$.

Then there exists  universal constants $C_1, C_2$ (depending on $\xi, C_0$) such that
if $\vert \al\vert >C_1\de^{(\xi-1)/\xi} \vert \beta\vert$
and $|\al| \le C_0 |\beta|$,
then
 $$ K_{\Om_\xi} (p_\de, \nu) \ge C_2 \frac{\vert \al\vert}{\de^{1-\frac1{2\xi}}},\quad \forall \de>0.$$
\end{lemma}

\begin{proof}
 We need an elementary lemma about the growth of holomorphic functions.

 \begin{lemma}
 \label{realf}
 Let $f_0(z)=\sum_{k \ge 1} a_k z^k$ be a holomorphic function on $\D$. Then
$$
M(r):=\sup_{\vert t\vert =r} \Re f_0(t) \ge \frac{\vert a_1 r\vert}2,\quad  \forall r \in (0,1).
$$
\end{lemma}
\begin{proof}
First
$$N(r):=\sup_{\vert t\vert =r} \vert f_0(t)\vert
 \ge \int_{\vert t\vert =r} \left| \sum_{k \ge 1} a_k t^k \right| \frac{dt}{2\pi}
 =r\int_{\vert t\vert =r} \left| a_1 +\sum_{k \ge 2}a_kt^{k-1}\right| \frac{dt}{2\pi}$$
$$ \ge r \Big \vt \int_{\vert t \vert=r}
(a_1+\sum_{k \ge 2}a_kt^{k-1} )\frac{dt}{2\pi} \Big \vert=\vert a_1 r\vert.
$$
Next, fix $r \in (0,1)$. For $r' \in (0, r)$,
by Borel-Caratheodory's theorem (note that $f_0(0)=0$) we obtain
$$M_r \ge \frac{N_{r'} (r-r')}{2r'} \ge \frac{\vert a_1\vert}2 (r-r').$$
Letting $r' \to 0$, we get the lemma.
\end{proof}

Returning to the lower estimate for $\Om_\xi,$
we may assume that $\beta=1$, $|\al|\le C_0$.
Consider an arbitrary analytic disk $\Phi=(f,g): \D \to \Om_\xi$ such that
\begin{equation}
\label{dieu3}
\Phi (0)=p_{\de}, \Phi' (0) =\la \nu. 
\end{equation}
Let's expand $f, g$ into Taylor series
$$f (t)= -\de +\la\al t+ a_2 t^2+\cdots , g (t)=\la t+\tilde g(t).$$
By the Schwarz Lemma and Cauchy inequality, we can see that
$$\vert \tilde g(t)\vert \le 2 \vert t\vert^2,\quad \forall \vert t\vert<1/2.\
$$
On a circle $\vert t\vert =r, r<1/2$, by the lemma above we have
$$\sup_{\vert t\vert =r} \Re f(t) \ge \frac{\vert \la r \vert}2 \vert \al\vert  -\de.$$
In view of the estimate on $\tilde g(t)$
 and convexity of the function $x^t, x>0, t\ge 1,$ we get
$$
\sup_{\vert t\vert =r} \vert g(t) \vert^\xi \le  2^{\xi-1} (\vert \la r \vert^\xi+2^\xi r^{2\xi}).
$$
Likewise,
$$
\sup_{\vert t\vert =r} \vert \Im f(t) \vert^\xi \le  2^{\xi-1} (\vert C_0 \la r \vert^\xi+2^\xi r^{2\xi}).
$$
Combining these estimates, we obtain the following basic inequality from which we will deduce a contradiction.
\begin{equation}
\label{basic}
\va (r):= 2^{2\xi+1} r^{2\xi}+2^\xi (1+C_0^\xi) \vert \la\vert ^\xi r^\xi- \vert \la \al\vert r+2 \de>0,\quad \forall 0<r<1/2.
\end{equation}
We have
\begin{equation}
\label{dieu4}
\va' (r)=\xi 2^{2\xi +2}r^{2\xi-1}+\xi 2^\xi (1+C_0^\xi) \vert \la\vert^\xi r^{\xi-1}-\vert \la \al \vert.
\end{equation}
Notice that
$$\va' (0)<0, \va'(1/2)>8\xi-\vert \la \al \vert >8-\vert \la \al \vert>0,$$ where the last inequality follows from the Schwarz Lemma.
Moreover, since $\xi >1$ we have $\va'' (r)>0$ for every $r>0$, so the equation $\va' (r)=0$ has a {\it unique} root $r_0 \in (0,1/2)$.
Now we have
\begin{equation}
\label{dieu5}
2\xi \va (r)= r\va'(r) +\psi (r),
\end{equation}
where
\begin{equation}
\label{dieu6}
\psi (r) =\xi 2^\xi (1+C_0^\xi) \vert \la \vert^\xi r^\xi-(2\xi-1)\vert \la \al \vert r+4\de \xi.
\end{equation}
Since $\va' (r_0)=0$, from \eqref{basic}, \eqref{dieu5} we infer that $\psi (r_0)>0.$
It also follows from \eqref{dieu4} that
$$\xi 2^\xi (1+C_0^\xi) \vert \la \vert^\xi r_0^{\xi-1} <\vert \la \al\vert.$$
Therefore
$$\xi 2^\xi \vert \la\vert^\xi (1+C_0^\xi) r_0^\xi <\vert \la \al \vert r_0.$$
Since $\psi (r_0)>0$, from \eqref{dieu6} and the above inequality we get
$$\vert \la \al \vert r_0 <\frac{2\xi}{\xi -1} \de.$$
Thus
$$r_0 <r_1:=\frac{2\xi}{\xi -1}(\frac{\de}{\vert \la \al\vert}).$$
This implies that
\begin{equation}
\label{dieu7}
0<\frac1{\xi 2^\xi} \va' (r_1)=2^{\xi+2} r_1^{2\xi-1} + (1+C_0^\xi) \vert \la\vert^\xi r_1^{\xi-1} -\frac{\vert \la \al\vert}{\xi 2^\xi}.
\end{equation}
Now we can choose $C_1>0$ depending only on $\xi$ and $C_0$
such that if $\vert \al \vert >C_1 \de^{(\xi-1)/\xi}$ then
\begin{equation}
\label{dieu8}
\vert \la\vert^\xi r_1^{\xi-1} <\frac12 \frac{\vert \la \al\vert}{\xi 2^\xi}.
\end{equation}
Putting \eqref{dieu7} and \eqref{dieu8} together, we get
$$\frac12 \frac{\vert \la \al\vert}{\xi 2^\xi} < 2^{\xi+1} r_1^{2\xi-1}.$$
Rearranging this inequality, we obtain
$$\vert \la \al\vert <C_2 \de^{(2\xi -1)/2\xi},$$
where $C_2>0$ depends only on $\xi$.
The desired lower bound follows.
\end{proof}

\section{Properties of the new pseudometric}
\label{propnew}

We list some properties of $\wdt K_D$
similar to those  of $K_D.$

\begin{proposition}\label{list} Let $D\subset\C^n$ and
$G\subset\C^m$ be domains.
\smallskip

(i) If $f\in\OO(D,G),$ then $\wdt K_D(z;X)\ge\wdt K_G(f(z);f_{\ast,z}(X)).$
\smallskip

(ii) $\wdt K_{D\times G}((z,w);(X,Y))=\max\{\wdt K_D(z;X),\wdt K_G(w;Y)\}.$
\smallskip

(iii) If $(D_j)$ is an exhaustion of $D$ by domains in $\C^n$
(i.e. $D_j\subset D_{j+1}$ and $\cup_j D_j=D$) and
$D_j\times\C^n\ni(a_j,X_j)\to (a,X)\in D\times\C^n,$ then
$$\limsup_{j\to\infty}\wdt K_{D_j}(a_j;X_j)\le\wdt K_D(a;X).$$
In particular, $\wdt K_D$ is an upper semicontinuous function.
\end{proposition}

\noindent{\it Proof.}
Denote by $\mathcal E(P)$ the envelope of holomorphy of a
domain $P\subset\C^k.$

 (i) If $k=\mbox{rank} f_{\ast,z},$
then $f_{\ast,z}(I_z K_D)\subset I_{f(z)} K_G$ is a balanced domain in $\C^k$ with
$f_{\ast,z}(\E(I_z K_D))$ as the envelope of holomorphy. It follows that
$f_{\ast,z}(\E(I_z K_D))\subset\E(I_{f(z)}K_G)$ which finishes the proof.
\smallskip

(ii) The Kobayashi metric has the product property
$$K_{D\times G}((z,w);(X,Y))=\max\{\wdt K_D(z;X),\wdt
K_G(w;Y)\},\quad\mbox{i.e.}$$
$$I_{(z,w)}K_{D\times G}=I_z K_D\times I_w
K_G.$$
 Then
$$\E(I_{(z,w)}K_{D\times G})=\E(I_z K_D)\times\E(I_w
K_G),$$ i.e. $\wdt K$ has the product property.
\smallskip

(iii) The case $X=0$ is trivial. Otherwise, after an unitary
transformation, we  may assume that all the components $X^k$ of
$X$ are non-zero. Set
$$\Phi_j(z)=(a^1+\frac{X^1}{X^1_j}(z^1-a^1_j),\dots,a^n+\frac{X^n}{X^n_j}(z^n-a^n_j)),
\quad j\gg 1.$$ We may find $\eps_j\searrow 0$ such that if
$G_j=\{z\in\C^n:\Bbb B_n(z,\eps_j)\subset D\},$ then
$G_j\subset\Phi_j(D_j).$ It follows that $\wdt K_{G_j}(a;X)\ge\wdt
K_{D_j}(a_j;X_j).$

Further, since $K_{G_j}\searrow K_D$ pointwise, it follows that
$I_a K_{G_j}\subset I_a K_{G_{j+1}}$ and $\cup_j I_a K_{G_j}=I_a
K_D.$ Then
$$\E(I_a K_{G_j})\subset\E(I_a K_{G_{j+1}})\mbox{ and
}\cup_j\E(I_a K_{G_j})=\E(I_a K_G).$$ Hence $\wdt
K_{D_j}(a_j;X_j)\le \wdt K_{G_j}(a;X)\searrow\wdt K_D(a;X)$
pointwise.\qed
\smallskip

\noindent{\bf Remark.} The above proof shows that Proposition
\ref{list}, (i) and (ii) remain true for complex manifolds.

To see (iii), note that it is known to hold with $K$ instead of $\wdt K$
(see the proof of
\cite[Proposition 3]{Roy}.

Moreover, any balanced domain can be exhausted by bounded balanced
domains with continuous Min\-kow\-ski functions  (see \cite[Lemma
4]{Nik-Pfl1}). Let $(E_k)$ be such an exhaustion of $I_a K_D.$
Then, by continuity of $h_{E_k}$,  for any $k$ there is a $j_k$ such that
$E_k\subset I_{a_j} K_{D_j}$ for any $j>j_k$. Hence, if we denote by
$h_k$  the Minkowski function of $\E(E_k)$, which is
upper semi-continuous,
$$\limsup_{j\to\infty}\wdt
K_{D_j}(a_j;X_j)\le\limsup_{j\to\infty}h_k(X_j)\le h_k(X).$$
  It remains to use
that $h_k(X)\searrow\wdt K_D(a;X).$

 Another way to see (iii) for
manifolds is to use the case of domains and the standard approach
in \cite[p.~2]{Nik-Pfl2} (embedding in $\C^N$).

\begin{proposition}\label{cont} Let $D\Subset\C^n$ be a pseudoconvex
domain with $\mathcal C^1$-smooth boundary. Let $(D_j)$ be a sequence of
bounded domains in $\C^n$ with $D\subset D_{j+1}\subset D_j$ and
$\cap_j D_j\subset\overline D.$ If $D_j\times\C^n\ni(z_j,X_j)\to
(z,X)\in D\times\C^n,$ then $\wdt K_{D_j}(z_j;X_j)\to\wdt K_D(z;X).$ In particular,
$\wdt K_D$ is a continuous function.
\end{proposition}

\noindent{\bf Remark.} It is well-known that any bounded pseudoconvex domain with
$\mathcal C^1$-smooth boundary is taut (i.e.~$\OO(\D,D)$ is a normal family).
It is unclear whether only the tautness of $D$ implies the continuity of
$\wdt K_D$ ($K_D$ has this property).
\smallskip

\noindent{\it Proof.} In virtue of Proposition \ref{list} (iii), we have only to show that
$$\liminf_{j\to\infty}\wdt K_{D_j}(z_j;X_j)\ge\wdt K_D(z;X).$$ Using the approach
in the proof of Proposition \ref{list} (iii), we may find another
sequence $(G_j)$ of domains with the same properties as $(D_j)$
such that  $\wdt K_{D_j}(z_j;X_j)\ge\wdt K_{G_j}(z;X).$ It follows
from the proof of \cite[Proposition 3.3.5 (b)]{Jar-Pfl1} that
$K_{G_j}\nearrow K_D$ pointwise and then $\cap_j I_z
K_{G_j}\subset cI_z K_D$ for any $c>1.$ Hence $\cap_j\E(I_z
K_{G_j})\subset c\E(I_z K_D)$ which completes the proof.$\qed$

\end{document}